\documentclass{amsart}

\usepackage{amssymb}

\usepackage{graphicx}
\usepackage{subcaption}
\usepackage{tikz}

\newtheorem{theorem}{Theorem}[section]
\newtheorem{lemma}[theorem]{Lemma}
\newtheorem{corollary}[theorem]{Corollary}
\newtheorem{proposition}[theorem]{Proposition}

\theoremstyle{definition}
\newtheorem{definition}[theorem]{Definition}

\theoremstyle{remark}

\numberwithin{equation}{section}

\DeclareMathOperator{\row}{row}
\DeclareMathOperator{\clmn}{col}

\DeclareMathOperator{\rnk}{rank}
\DeclareMathOperator*{\argmax}{argmax}

\newcommand{\norm}[1]{\left\Vert #1 \right\Vert}

\newcommand{\Prb}[1]{\mathbb{P}\left[ #1 \right]}
\newcommand{\E}[1]{\mathbb{E}\left[ #1 \right]}

\newcommand{\cond}[2]{\mathbb{E}\left[\left. #1 \right\vert #2 \right]}
\newcommand{\condPrb}[2]{\mathbb{P}\left[\left. #1 \right\vert #2 \right]}

\newcommand{\linspan}[1]{\mathrm{span}\left[ #1 \right]}

\begin{document}

\title{Convergence of Adaptive, Randomized, Iterative Linear Solvers}


\author[Patel]{Vivak Patel}
\address{Department of Statistics, University of Wisconsin -- Madison, Madison, Wisconsin 53706}
\curraddr{}
\email{vivak.patel@wisc.edu}
\thanks{Vivak Patel is supported by the Wisconsin Alumni Research Foundation}

\author[Jahangoshahi]{Mohammad Jahangoshahi}
\address{Susquehanna International Group, Bala Cynwyd, Pennsylvania 19004}
\curraddr{}
\email{mjahangoshahi@uchicago.edu }
\thanks{}

\author[Maldonado]{Daniel Adrian Maldonado}
\address{Mathematics and Computer Science, Argonne National Laboratory, Lemont, Illinois 60439}
\email{maldonadod@anl.gov}
\thanks{Daniel Adrian Maldonado is supported by DOE Contract DE- AC02-06CH11347}


\date{}

\dedicatory{}

\begin{abstract}
Deterministic and randomized, row-action and column-action linear solvers have become increasingly popular owing to their simplicity, low computational and memory complexities, and ease of composition with other techniques. Moreover, in order to achieve high-performance, such solvers must often be adapted to the given problem structure and to the hardware platform on which the problem will be solved. Unfortunately, determining whether such adapted solvers will converge to a solution has required equally unique analyses. As a result, adapted, reliable solvers are slow to be developed and deployed. In this work, we provide a general set of assumptions under which such adapted solvers are guaranteed to converge with probability one, and provide worst case rates of convergence. As a result, we can provide practitioners with guidance on how to design highly adapted, randomized or deterministic, row-action or column-action linear solvers that are also guaranteed to converge.
\end{abstract}

\maketitle

\section{Introduction}
\label{section-introduction}
Iterative linear solvers are often preferred for solving large-scale linear systems, as they can take advantage of problem structure such as sparsity or bandedness, require inexpensive floating point operations, and can be readily paired with preconditioning techniques \cite[see preface]{saad2003}. While such iterative linear solvers as Conjugate Gradients (CG) and the Generalized Minimal Residual method (GMRES) are still dominant solvers in practice, randomized row-action \cite{karczmarz1937,agmon1954,motzkin1954,strohmer2009} and column-action iterative solvers \cite{leventhal2010,zouzias2013} have been growing in interest for several reasons: they (usually) require very few floating point operations per iteration \cite{gordon1970,censor1981}; they have low-memory footprints \cite{lent1976}; they can readily be composed with randomization techniques to quickly produce approximate solutions \cite{strohmer2009,leventhal2010,wallace2014,gower2015,ma2015,bai2018,haddock2019,PAJAMA2019}; they can be used for solving systems constructed in a streaming fashion (e.g., \cite{needell2016}), which supports emerging computing paradigms (e.g., \cite{mills2020}); and, just like the more popular iterative Krylov solvers, they can be parallelized, preconditioned or combined with other linear solvers \cite{sameh1978,nutini2016,du2020,richtarik2020};

Unfortunately, vanilla forms of these row-action and column-action iterative solvers do not consider problem structure or make any hardware considerations, which often results in untenable inefficiencies \cite{nutini2016}. 
To illustrate, consider the Kaczmarz method with uniform row sampling with replacement as applied to the systems whose coefficient matrices are pictorially represented in Figures \ref{fig-ortho-system} and \ref{fig-split-system}. In the coefficient matrix given by Figure \ref{fig-ortho-system}, the aforementioned Kaczmarz method would require over twice as many iterations (in expectation) in comparison to a row-action method that accounted for the orthogonal structure \cite{nutini2016}.

\begin{figure}[ht!]

        \centering
        \begin{tikzpicture}
\def\div{5};
\def\dim{4};

\draw[step=0.2cm,gray,very thin, opacity=0.5] (0,0) grid (\dim,\dim);

\pgfmathparse{\dim/(\div * 0.2 )};
\let\size\pgfmathresult;

\foreach \i in {1,...,\div}{
	\pgfmathparse{ \i/\div*100 }
	\let\col\pgfmathresult;

	\foreach \j in {1,...,\size}{
	\foreach \k in {1,...,\size}{
		\pgfmathparse{(\i-1)* (\dim/\div) + (\j-1)*0.2 + 0.1};
		\let\x\pgfmathresult;
		
		\pgfmathparse{(\i-1)* (\dim/\div) + (\k-1)*0.2 + 0.1};
		\let\y\pgfmathresult;
		
		\node at (\x,\y) [circle, fill=orange!\col!blue, inner sep=1.5pt] {};
		
	};	
	};
};

%
%
%
%

%
%




\end{tikzpicture}
        \caption{A coefficient matrix of a linear system with block structure. Blocks with the same color have colinear rows. Empty cells indicate zeros of the given matrix.}
        \label{fig-ortho-system}
\end{figure}
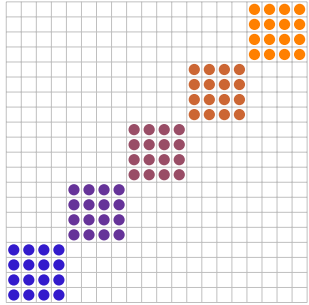

In the coefficient matrix given by Figure \ref{fig-split-system}, the aforementioned Kaczmarz procedure would require reading in a new group almost evey other iteration (in expectation), which would have detrimental Input and Output (I/O) costs. In comparison, a procedure which randomly cycled through the rows of a given group until the group's residual was reduced to a fraction of its original value and then moved on to a new group would require substantially lower I/O costs.
\begin{figure}[ht!]
        \centering
        \begin{tikzpicture}
\def\div{5};
\def\dim{4};

\draw[step=0.2cm,gray,very thin, opacity=0.5] (0,0) grid (3,\dim);

\pgfmathparse{\dim/(\div * 0.2 )};
\let\size\pgfmathresult;

\pgfmathparse{\size*\div};
\let\steps\pgfmathresult;

\foreach \i in {1,...,\div}{
	\pgfmathparse{ \i/\div*100 }
	\let\col\pgfmathresult;
	
	\pgfmathparse{\size+2}
	\let\start\pgfmathresult;
	\foreach \j in {\start,...,\steps}{
	\foreach \k in {1,...,\size}{
		\pgfmathparse{(\j-\start)*0.2 + 0.1};
		\let\x\pgfmathresult;
		
		\pgfmathparse{(\i-1)* (\dim/\div) + (\k-1)*0.2 + 0.1};
		\let\y\pgfmathresult;
		
		\node at (\x,\y) [circle, fill=orange!\col!blue, inner sep=1.5pt] {};
		
	};	
	};
};

\def\cpu{5};

\foreach \p in {2,...,\cpu}{
	\pgfmathparse{\dim/(\cpu)*(\p-1)};
	\let\y\pgfmathresult;
	\draw[dashed,thick] (-2,\y) -- (3.5,\y);
};


\foreach \p in {1,...,\cpu}{
	\pgfmathparse{\dim - \dim/(\cpu)*(\p-0.5)};
	\let\y\pgfmathresult;
	\node at (-1,\y) {Group \p};
};
\end{tikzpicture}
        \caption{A coefficient matrix of a linear system too large to store in memory. The system is split into groups (by color) that are just small enough to fit in memory.}
        \label{fig-split-system}
\end{figure}
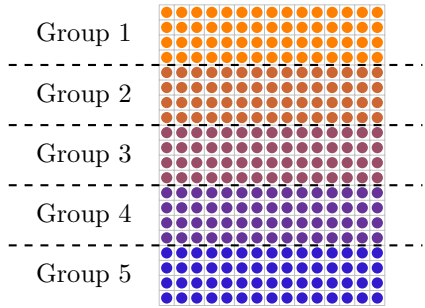

As these examples show, it is easy to imagine a plethora of adaptive variants of row-action and column-action methods, both random and deterministic, that would take advantage of the unique problem structures and hardware considerations to readily increase the speed-to-solution. Unfortunately, heretofore, any such adaptive variants have required their own unique analyses (e.g., \cite{agmon1954,motzkin1954,strohmer2009,leventhal2010,gower2015,bai2018,haddock2019,steinerberger2020}). As a result, rigorous, adaptive iterative solvers have been difficult to develop and deploy. 

To address this issue, building from our previous results \cite{PAJAMA2019}, we specify a set of general conditions for such solvers under which we can guarantee convergence with probability one (w.p.1.). Moreover, we are also able to provide a worst case rate of convergence, which generalizes the theory for deterministic solvers \cite{agmon1954,motzkin1954} and complements the specialized mean-squared analyses for certain random solvers \cite{strohmer2009,leventhal2010,gower2015,bai2018,haddock2019,steinerberger2020}. Thus, we are able to provide practitioners with a set of guiding principles to readily develop and deploy solvers that are highly adapted to their problem's structure and to their hardware platform, while also guaranteeing convergence.

The remainder of this work is organized as follows. In \S \ref{section-adaptive-method}, we define general, adaptive, random and deterministic, row-action and column-action iterative solvers, and present a unified recursive equation for these methods, which is closely related to \cite{ma2015}. 
In \S \ref{section-notation}, we collect some notation that appears throughout the work. 
In \S \ref{section-strong-convergence}, we prove convergence with probability one and provide a rate of worst case rate of convergence.
In \S \ref{section-applications}, we demonstrate how our results can be applied to a broad array of row-action and column-action solvers.
In \S \ref{section-conclusion}, we conclude this work.

\section{The Generic Adaptive Method}
\label{section-adaptive-method}
For a given linear system, let $A \in \mathbb{R}^{n \times d}$ denote a coefficient matrix and $b \in \mathbb{R}^n$ denote the constant vector. A generic adaptive procedure, randomized or deterministic, starts with an initial iterate, $x_0 \in \mathbb{R}^d$, and some auxiliary information, $\zeta_{-1}$. The auxiliary information can be, for instance, the indices of unused equations for a column-action method that samples columns without replacement. With this initial information, a row-action adaptive procedure (e.g., randomized Kaczmarz) runs the iteration
\begin{subequations} \label{row-action-complete}
\begin{equation}
w_k, \zeta_{k} = \varphi(A, b, \lbrace x_j: j \leq k \rbrace, \lbrace w_j : j < k \rbrace, \lbrace \zeta_j : j < k \rbrace ) \\
\end{equation}
\begin{equation} \label{row-action}
x_{k+1} = x_k + \frac{A' w_k w_k'(b - Ax_k)}{\norm{A'w_k}_2^2},
\end{equation}
\end{subequations}
where $\varphi$ is a function specific to the adaptive procedure and can be random; and $w_k$ takes values in $\mathbb{R}^n$. Similarly, a column-action adaptive procedure (e.g., randomized Coordinate Descent) runs the iteration
\begin{subequations} \label{column-action-complete}
\begin{equation}
w_k, \zeta_{k} = \varphi(A, b, \lbrace x_j: j \leq k \rbrace, \lbrace w_j : j < k \rbrace, \lbrace \zeta_j : j < k \rbrace ) \\
\end{equation}
\begin{equation} \label{column-action}
x_{k+1} = x_k + \frac{w_k w_k'A'(b - Ax_k)}{\norm{Aw_k}_2^2},
\end{equation}
\end{subequations}
where $\varphi$ is a function specific to the adaptive procedure and can be random; and $w_k$ takes values in $\mathbb{R}^d$.

For the purposes of analysis, when $Ax = b$ is consistent---that is,
\begin{equation} \label{assumption-consistency}
\exists x \in \mathbb{R}^d : Ax = b \text{~---},
\end{equation}
we will be able to summarize \eqref{row-action} and \eqref{column-action} using a single recursive relationship. Namely,
\begin{equation} \label{general-recursion}
y_{k+1} = y_k - M'w_k \frac{ w_k'M y_k}{ \norm{ M' w_k}_2^2 }, 
\end{equation}
where,
\begin{enumerate}
\item for \eqref{row-action}, $y_k = x_k - x^*$ with $x^*$ as the projection of $x_0$ onto the solution set, and $M=A$; while, 
\item for \eqref{column-action}, $y_k = Ax_k - b$ and $M = A'$.
\end{enumerate}

From \eqref{general-recursion}, we conclude, the projections of $\lbrace y_k \rbrace$ onto the null space of $M$ are constant. In the case of \eqref{row-action}, we interpret this as the component of $x_0$ in the null space of $A$ remains unchanged by the procedure. In the case of \eqref{column-action}, we have no meaningful interpretation since the system is assumed to be consistent \eqref{assumption-consistency}.

\section{Notation}
\label{section-notation}
Here, we collect relevant notation and definitions. 
For all $j+1 \in \mathbb{N}$ and $k \in [1, j+1] \cap \mathbb{N}$, let $\mathcal{F}_k^j = \sigma( \zeta_{j-k}, x_{j-k+1},w_{j-k + 1}, \ldots,w_{j-1}, \zeta_{j-1}, x_j )$ represent the information contained at iteration $j$ up to the preceding $k$ iterations. 
Note, $\mathcal{F}_1^j = \sigma( \zeta_{j-1}, x_j )$.
For short hand, let $\mathcal{F}^j = \mathcal{F}_{j+1}^j$.
For completeness, let $\mathcal{F}^j_0$ be the trivial $\sigma$-algebra.

Now, let 
\begin{equation}
\chi_j = \begin{cases}
1 & w_j'My_j \neq 0 \\
0 & \textrm{otherwise},
\end{cases}
\end{equation}
which has several interpretations. For example, $\chi_j$ can be interpreted as whether $y_j$ is orthogonal to the plane of projection, specified by $M'w_j$. As another example, $\chi_j$ can be interpreted as whether $y_j$ is orthogonal to the search direction $M'w_j$.

Using $\chi_j$, we can define the subspace generated by the nontrivial search directions from iteration $j$ to iterations $j+k$. Specifically, for $j+1 \in \mathbb{N}$ and $k \geq 0$, define 
\begin{equation}
\Phi_k^j = \left\lbrace M'w_j\chi_j/\norm{ M'w_j}_2 ,\ldots, M'w_{j+k}\chi_{j+k}/\norm{M'w_{j+k}}_2  \right\rbrace.
\end{equation}
An important related subspace is the one generated by the iterates of the procedure. Namely, for $j+1 \in \mathbb{N}$ and $k \geq 0$, define 
\begin{equation}
V_k^j = \linspan{ y_{j}, y_{j+1},\ldots, y_{j+k} }.
\end{equation}
We will see that the interplay between $V_k^j$ and $\Phi_k^j$ will be key to demonstrating the convergence of the procedure. To establish this relationship, we will need to define two stopping times. Let 
\begin{equation}
\nu(j) = \min \left\lbrace k \geq 0: y_{j+k+1} \in V_k^j, \chi_{j+k} \neq 0 \right\rbrace,
\end{equation} and, as a related quantity, 
\begin{equation} 
s(j) = \min \lbrace k \geq 0: \chi_{j+k} \neq 0 \rbrace,
\end{equation}
for all $j+1 \in \mathbb{N}$.
Finally, for a vector space $W$, $V \subsetneq W$ indicates that $V$ is a proper subspace of $W$, and $V \subset W$ indicates that $V$ can be any subspace of $W$.

\section{Strong Convergence of the Generic Adaptive Method}
\label{section-strong-convergence}
The key inequality on which our results depend is an extension of Meany's inequality \cite{meany1969}, derived in \cite[Theorem 4]{PAJAMA2019}, and stated here.
\begin{theorem} \label{theorem-meany}
Fix $j+1,k+1 \in \mathbb{N}$. Let $\mathcal{G}$ denote all matrices $G$ whose columns are maximal linearly independent subsets of $\Phi_k^j$. Then, $\forall y \in \linspan{ \Phi_k^j}$,
\begin{equation} \label{equation-meany}
\begin{aligned}
& \norm{ \left( I - \frac{M' w_{j+k} w_{j+k}'M}{\norm{M' w_{j+k}}_2^2}\chi_{j+k} \right) \cdots \left(  I - \frac{M' w_{j} w_{j}'M}{\norm{M' w_{j}}_2^2}\chi_j  \right) y }_2^2 \\
&\quad\quad\quad\quad\quad\quad\quad \leq \left[ 1 - \min_{G \in \mathcal{G} } \det(G'G) \right] \norm{y}_2^2.
\end{aligned}
\end{equation}
\end{theorem}

As an example of how we will apply this result, we need to find the first iterate $j$, when $\mathrm{span}[\Phi^j_0]$ contains $y_0$. Applying Meany's inequality ensures a certain degree of reduction between $y_0$ and $y_{j+1}$ (in norm). To determine this value of $j$, we will need to relate $\Phi^j_k$ and $V_k^j$ and allow for the possibility that $j$ can be random. This is the content of the next result.

\begin{lemma} \label{lemma-revisit-characterization}
Let $\xi$ be an arbitrary, finite stopping time with respect to $\lbrace \mathcal{F}^k \rbrace$. Under \eqref{assumption-consistency}, if $\nu(\xi)$ is finite, then $V_{\nu(\xi)}^\xi = \mathrm{span}[ \Phi_{{\nu}(\xi)}^\xi].$ Moreover, the nonzero elements of $\Phi_{\nu(\xi)}^\xi$ are linearly independent.
\end{lemma}
\begin{proof}
By \eqref{general-recursion}, $M' w_{\xi+k} \chi_{\xi+k} \in \linspan{ y_{\xi+k+1} - y_{\xi+k} } \subset V_{\nu(\xi)}^\xi$ for all integers $k \in [0, \nu(\xi)]$. Therefore, $\linspan{ \Phi_{\nu(\xi)}^\xi} \subset V_{\nu(\xi)}^\xi$. We now prove set inclusion in the other direction. 

Let $\Phi = \lbrace y_{\xi + \nu(\xi) }, M'w_\xi \chi_\xi, M' w_{\xi+1} \chi_{\xi+1},\ldots,$ $M'w_{\xi+\nu(\xi)-1} \chi_{\xi+\nu(\xi)-1} \rbrace$. By \eqref{general-recursion}, $\linspan{ y_{k+1}, M'w_k \chi_k }$ contains $y_k$. By this fact and since $\Phi_{\nu(\xi)}^\xi \subset V_{\nu(\xi)}^\xi$, $V_{\nu(\xi)}^\xi = \linspan{ \Phi }$. Therefore, we can prove the result if we can replace $y_{\xi + \nu(\xi)}$ with the direction $M'w_{\xi + \nu(\xi)} \chi_{\xi + \nu(\xi)}$ in the generating set $\Phi$.

We first note that $y_{\xi + \nu(\xi)} \neq 0$. Indeed, if this were true, then $y_{\xi + \nu(\xi) + 1} = y_{\xi + \nu(\xi)}$, which would contradict $\chi_{\xi + \nu(\xi)} \neq 0$ in the definition of $\nu$ under the assumption that $\nu(\xi)$ is finite.

Next, we note that the nonzero terms in $\Phi$ are linearly independent. Suppose this were not true, then $y_{\xi + \nu(\xi)}$ is in the span of the remaining terms in $\Phi$. Since $\Phi_{\nu(\xi)-1}^\xi \subset V_{\nu(\xi)-1}^\xi$, this would imply that $y_{\xi + \nu(\xi) } \in V_{\nu(\xi)-1}^\xi$, which contradicts the minimality of $\nu(\xi)$.

Let $r = \dim ( V_{\nu(\xi)}^\xi )$. Given the linear independence of the nonzero terms in $\Phi$, we can use the Gram-Schmidt procedure to construct the orthogonal set of vectors, $\lbrace y_{\xi + \nu(\xi)}, \phi_1,\ldots,\phi_{r-1} \rbrace$, whose span is that of $\Phi$. Since $y_{\xi + \nu(\xi) + 1} \in V_{\nu(\xi)}^\xi$ by construction, there exist scalars $\lbrace c_i : i = 0,\ldots,r-1 \rbrace$ such that
\begin{equation}
c_0 y_{\xi+ \nu(\xi)} + \sum_{j=1}^{r-1} c_j \phi_j = y_{\xi+\nu(\xi)+1} = y_{\xi + \nu(\xi)} - M' w_{\xi + \nu(\xi)} \frac{w_{\xi + \nu(\xi)}'M y_{\xi + \nu(\xi)}}{\norm{ M' w_{\xi + \nu(\xi)}}_2^2}.
\end{equation}
If $c_0 \neq 1$, then $y_{\xi+ \nu(\xi)} \in \linspan{ \phi_1,\ldots,\phi_{r-1}, M' w_{\xi + \nu(\xi) }}$, which implies the result. If $c_0 = 1$, then $M'w_{\xi + \nu(\xi)}$, is in the span of $\lbrace \phi_1,\ldots,\phi_{r-1} \rbrace$, which is orthogonal to $y_{\xi+\nu(\xi)}$. This implies that $\chi_{\xi+\nu(\xi)} = 0$, which contradicts the definition of $\nu(\xi)$. The first part of the result follows. The second part follows, by the first part and the linear independence of the nonzero terms in $\Phi$. 
\end{proof}

Thus, combining Lemma \ref{lemma-revisit-characterization} and Theorem \ref{theorem-meany}, we can guarantee a reduction (in norm) from $y_{\xi}$ to $y_{\xi+\nu(\xi)+1}$, \textit{so long as $\nu(\xi)$ is well-behaved}. The following definitions and subsequent lemma ensure that $\nu(\xi)$ is well-behaved. Specifically, the first definition specifies that the procedure forgets the past after a fixed number of iterations.

\begin{definition}[$N$-Markovian] \label{assumption-n-markovian}
An adaptive method is $N$-Markovian for $N\in \mathbb{N}$ if, for any measurable sets $W$ and $Z$ with respect to $w_k$ and $\zeta_k$, 
\begin{equation} 
\condPrb{ w_{k} \in W, \zeta_k \in Z }{ \mathcal{F}^k} = \condPrb{ w_k \in W, \zeta_k \in Z}{ \mathcal{F}_N^k}.
\end{equation}
\end{definition}

Note, the definition of $N$-Markovian does not explicitly include the case where $w_k$ is generated independently of $x_k$. However, the definition implicitly includes this case since such a situation is a special case of being $1$-Markovian. More generally, a procedure that is $N$-Markovian is also $(N+1)$-Markovian.

The next definition specifies that the procedure actually updates the iterate within this window with some nonzero probability. Indeed, the next definition prevents the iterates from being confined to a subspace that is distinct from the solution set.

\begin{definition}[Exploratory] \label{assumption-nontrivial-decay}
An adaptive, $N$-Markovain method is exploratory if
\begin{equation} 
\exists \pi \in (0,1],~ \forall V \subsetneq \row(M) : \sup_{\substack{y_0 \in V \setminus \lbrace 0 \rbrace \\ \zeta_{-1}}} \condPrb{ \cap_{j=0}^{N-1} \lbrace M'w_j \perp V \rbrace  }{\mathcal{F}^0} \leq 1 - \pi.
\end{equation}
\end{definition}

As we will show in \S \ref{section-applications}, both of these definitions are verifiable for a host of procedures, which allows them to be used in the design and development of adaptive, randomized linear solvers. For now, we show the consequence of a procedure that satisfies these two definitions.

\begin{lemma} \label{lemma-finite-revisit-time}
Let $\xi$ be an arbitrary, finite stopping time with respect to $\lbrace \mathcal{F}^k \rbrace$, and let $\mathcal{F}^{\xi}$ denote the stopped $\sigma$-algebra. Under \eqref{assumption-consistency}, Definition \ref{assumption-n-markovian} with $N \in \mathbb{N}$, and Definition \ref{assumption-nontrivial-decay} with $\pi \in (0,1]$, if $y_{\xi} \neq 0$, then $\nu(\xi)$ is finite and $\cond{ \nu(\xi) }{\mathcal{F}^\xi} \leq  N \times \rnk(M) / \pi$ with probability one.
\end{lemma}
\begin{proof}
Since $y_{\xi} \neq 0$, Definition \ref{assumption-nontrivial-decay} implies $\condPrb{s(\xi) \geq \ell N}{\mathcal{F}^\xi} \leq (1 - \pi)^\ell$. Comparing to a geometric process, $\cond{ s(\xi) }{\mathcal{F}^\xi} \leq N/\pi$. Accordingly, for any $j \in \mathbb{N}$, if $y_{\xi+ s_1 + \cdots + s_{j-1}} \neq 0$, we can define $s_j = s( \xi+s_1 + \cdots + s_{j-1} )$. Then, 
\begin{equation}
\cond{ s_1 + \cdots + s_j }{\mathcal{F}^\xi} \leq j N/\pi.
\end{equation}

Moreover, for any $\tau \in \lbrace \sum_{i=1}^j s_i : j \in \mathbb{N} \rbrace$ either (Case 1) $V_{\tau}^\xi = V_{\tau+1}^\xi$ or (Case 2) $\dim( V_{\tau}^\xi )+1 = \dim ( V_{\tau+1}^\xi)$. Given that $V_{k}^\xi \subset \row(M)$, we see that the second case can only happen at most $\rnk(M)-1$ times before the first case must be true.
With this fact and by definition of $\nu(\xi)$, $\nu(\xi) \in \lbrace \sum_{i=1}^j s_i : j = 1,\ldots, \rnk(M) \rbrace$. Thus, $\cond{\nu(\xi)}{\mathcal{F}^\xi} \leq N \rnk(M) / \pi$.
\end{proof}

We now combine the above results to characterize the behavior of $\lbrace y_k \rbrace$ for a general adaptive, randomized procedures. Note, when $y_k$ is zero or $\lbrace y_k \rbrace$ converges to zero, this is equivalent to $x_k$ being equal to, or converging to, a solution of the system, respectively.

\begin{theorem} \label{theorem-control}
Let $A \in \mathbb{R}^{n\times d}$ and $b \in \mathbb{R}^n$ be a consistent system \eqref{assumption-consistency}. Moreover, suppose $x_0 \in \mathbb{R}^d$ is not a solution to the system (i.e., $Ax_0 \neq b$) and $\lbrace x_k : k \in \mathbb{N} \rbrace$ is generated by either \eqref{row-action-complete} or \eqref{column-action-complete} satisfying Definition \ref{assumption-n-markovian} with $N \in \mathbb{N}$, and Definition \ref{assumption-nontrivial-decay} with $\pi \in (0,1]$. Let $\lbrace y_k \rbrace$ be defined as in \eqref{general-recursion}. Then, there exist stopping times, $\lbrace \tau_k \rbrace$, with $\E{ \tau_k } \leq k[1 + N\rnk(A)/\pi]$, and random variables $\lbrace \gamma_k \in [0,1) \rbrace$ such that $ \norm{ y_{\tau_{k+1}} }_2^2 \leq \gamma_k \norm{ y_{\tau_k} }_2^2$.
\end{theorem}
\begin{proof}
The proof proceeds by induction. Let $\tau_0 = 0$. For the induction hypothesis, we assume that $\E{ \tau_{j} }  \leq j[1 + N  \rnk(A)/\pi]$. There are two cases: $y_{\tau_j} = 0$ or $y_{\tau_j} \neq 0$. In the former case, $y_{\tau_j+k} = 0$ for all $k \in \mathbb{N}$. Therefore, we define $\tau_{j+k} = \tau_j +k$ and $\gamma_{j+k-1} =0$ for all $k \in \mathbb{N}$, and the result follows. In the second case, let $\tau_{j+1} = \nu( \tau_j) + \tau_j + 1$. By the induction hypothesis and Lemma \ref{lemma-finite-revisit-time}, $\E{ \tau_{j+1} } \leq (j+1) [1 + N\rnk(A)/\pi]$. Moreover, by  Lemma \ref{lemma-revisit-characterization}, $y_{\tau_{j+1}} \in V_{\nu(\tau_j)}^{\tau_j}$ and $V_{\nu(\tau_j)}^{\tau_j} = \linspan{ \Phi_{\nu(\tau_j)}^{\tau_j} }$. Therefore, by Theorem \ref{theorem-meany}, $\norm{y_{\tau_{j+1}}}_2^2 \leq \gamma_j \norm{ y_{\tau_j} }_2^2$, where $\gamma_j$ is given by the right hand side of \eqref{equation-meany} and is in $[0,1)$ by Hadamard's inequality. Thus, the result follows by induction.
\end{proof}

Note, in the proof of Theorem \ref{theorem-control}, we identify two cases: finite termination and an infinite sequence of $\lbrace \norm{y_k}_2 \rbrace$. In the latter case, we have not guaranteed convergence of the sequence, and, in order to do so, we must ensure the event of $\gamma_k \to 1$ as $k \to \infty$ has probability zero. While we will present a general way of ensuring that this event holds with probability zero, we begin with a more special situation that includes the case where $\lbrace w_k \rbrace$ are standard basis elements \cite{agmon1954,motzkin1954,strohmer2009,zouzias2013,bai2018,haddock2019,steinerberger2020}.

\begin{corollary} \label{corollary-finite-convergence}
Suppose the setting of Theorem \ref{theorem-control} holds. Moreover, suppose that $\lbrace M'w_k \rbrace$ belong to a finite set. Then, $\exists \gamma \in [0,1)$ such that with probability one and for all $k \in \mathbb{N}$,
\begin{equation}
\norm{ y_{\tau_{k}} }_2^2 \leq \gamma^k \norm{y_0}_2^2.
\end{equation}
\end{corollary} 
\begin{proof}
By Theorem \ref{theorem-control}, it is enough to show that there exists a $\gamma \in [0,1)$ such that $\Prb{ \gamma_k \leq \gamma } = 1$ for all $k+1 \in \mathbb{N}$. Let $\lbrace h_1,\ldots,h_q \rbrace$ denote the finite set in which $\lbrace M'w_k \rbrace$ take value. Let $\mathcal{H}$ be the set of all matrices $H$ whose columns are maximally linearly independent subsets of $\lbrace h_1/\norm{h_1}_2,\ldots,h_q/\norm{h_q}_2 \rbrace$. Then,
\begin{equation}
\gamma_k \leq 1 - \inf_{ H \in \mathcal{H}} \det(H'H).
\end{equation}
Since $\lbrace h_1,\ldots,h_q \rbrace$ is finite, $\mathcal{H}$ is finite. This implies that the infimum exists. So, we can define $\gamma$ to be the right hand side, and since all $H \in \mathcal{H}$ are have full column rank, $\gamma \in [0,1)$.
\end{proof}

If we do not have the finiteness assumed in Corollary \ref{corollary-finite-convergence}, we need to find another way to control $\lbrace \gamma_k \rbrace$ regardless of the evolution of $\lbrace y_k \rbrace$. The following definition provides one rather generic way of ensuring this behavior. 

\begin{definition}[Uniformly Controlled in Expectation] \label{assumption-det-control}
Suppose $\varphi$ is $N$-Markovian and exploratory for some $\pi \in (0,1]$. 
For any $y_0$, let $\tilde G$ denote the matrix whose columns are the normalized, unique (by Lemmas \ref{lemma-revisit-characterization} and \ref{lemma-finite-revisit-time}), maximal linearly independent subset of $\Phi_{\nu(N-1)}^{N-1}$. Then $\varphi$ is uniformly controlled in expectation if
\begin{equation} 
\exists g \in (0,1] : \inf_{\substack{y_0 \in \row(M) \setminus \lbrace 0 \rbrace \\ \zeta_{-1}}} \cond{ \det( \tilde G ' \tilde G )}{\mathcal{F}^0} \geq g ~ (w.p.1.).
\end{equation}
\end{definition}

Now, with this definition we can prove that the iterates will converge to zero and provide a rather coarse, limiting rate of convergence.

\begin{corollary} \label{corollary-convergence}
Suppose the setting of Theorem \ref{theorem-control} holds, and suppose Definition \ref{assumption-det-control} holds for some $g \in (0,1]$. Then, with probability one, for any $\delta \in ( 1-g,1)$, there exists a finite stopping time $L$ such that for any $\ell \geq L$,
\begin{equation}
\norm{ y_{\tau_{\ell N } }}_2^2 \leq \delta^\ell \norm{ y_0}_2^2.
\end{equation}
Consequently, $\Prb{ \lim_{k \to \infty} \norm{y_k}_2 = 0 } = 1$.
\end{corollary}
\begin{proof}
We begin by exploring some consequences of Definition \ref{assumption-det-control}. First, let $G_j$ denote the matrix whose columns are the normalized, unique (by Lemmas \ref{lemma-revisit-characterization} and \ref{lemma-finite-revisit-time}), maximal linearly independent subset of $\Phi_{\nu(\tau_j)}^j$. Then, by Definition \ref{assumption-det-control}, with probability one, 
\begin{equation}
\cond{\gamma_j}{\mathcal{F}^{\tau_j-N+1}} = 1 - \cond{ \det(G_j'G_j)}{\mathcal{F}^{\tau_j-N+1}} \leq 1 - g.
\end{equation}

Using this, we would like to control $\mathbb{E}[ \prod_{\ell=0}^j \gamma_{\ell}]$, but we cannot naively make use of conditional expectations given that we have no guarantee of conditional independence between, say, $\gamma_{j-1}$ and $\gamma_j$ because $\gamma_{j-1}$ is measurable with respect to $\mathcal{F}^{\tau_j}$. On the other hand, we do have conditional independence between $\gamma_{\ell}$ and $\gamma_j$ when $\tau_{\ell+1} \leq \tau_{j}-N+1.$

Recall $\tau_{j+1} = \tau_j + 1$ if $y_{\tau_j} = 0$ and $\tau_{j+1} = \tau_j + \nu(\tau_j) + 1$ if $y_{\tau_j} \neq 0$. By this construction, 
\begin{equation}
\tau_{\ell N + 1} \leq \tau_{\ell N + 2} - 1 \leq \cdots \leq \tau_{\ell N + N} - N + 1.
\end{equation}
Therefore, $\gamma_{(\ell + 1 )N}$ is conditionally independent of $\gamma_{\ell N}$ given $\mathcal{F}^{\tau_{ (\ell +1)N} - N + 1}$. Using this fact and Markov's Inequality, for any $m+1 \in \mathbb{N}$,
\begin{equation}
\begin{aligned}
&\Prb{ \prod_{\ell=0}^{m N } \gamma_{\ell} > \delta^{m+1} } 
\leq \Prb{ \prod_{\ell=0}^{m} \gamma_{\ell N} > \delta^{m+1}} \\
&\quad\quad\leq \frac{  \E{ \prod_{\ell=0}^m \gamma_{\ell N} } }{\delta^{m+1}} 
\leq \left( \frac{1-g}{\delta} \right)^{m+1}.
\end{aligned}
\end{equation}

Applying the Borel-Cantelli lemma, there exists a finite stopping time $L$ (depending on $\delta$) such that if $m \geq L$, then $\prod_{\ell=0}^{mN} \gamma_{\ell} \leq \delta^{m+1}$ with probability one. Then, using Theorem \ref{theorem-control} and $\norm{ y_{j+1}}_2 \leq \norm{y_j}_2$, for $m \geq L$, with probability one,
\begin{equation}
\norm{ y_{ \tau_{(m+1) N}} }_2^2 \leq \left( \prod_{\ell = 0}^{m N} \gamma_{\ell} \right) \norm{y_0}_2^2 \leq \delta^{m+1} \norm{y_0}_2^2.
\end{equation}
Moreover, since $\norm{y_{j+1}}_2 \leq \norm{y_j}_2$, convergence of a subsequence of $\lbrace \norm{y_j} \rbrace$ to zero with probability one implies that the sequence converges to zero with probability one. 
\end{proof}

\section{Applications}
\label{section-applications}
Here, we present an example of how the previous results can be used to demonstrate convergence of row-action and column-action solvers for linear systems that are consistent (i.e., satisfying \eqref{assumption-consistency}). We will include many examples from the literature, several novel procedures, and straightforward generalizations.

\subsection{Independent and Identically Distributed} We start by considering row-action and column-action solvers in which $\lbrace w_k \rbrace$ are independent and identically distributed, which includes randomized Kaczmarz \cite{strohmer2009,steinerberger2020}, randomized Coordinate Descent \cite{zouzias2013}, and more general randomized vector sketching methods \cite[\S3.2 with $B=I$]{gower2015}. We now specify the behavior of $\varphi$ and how our results can be applied.

For these independent methods, there is no auxiliary information to track, so we can let $\zeta_k = \emptyset$ for all $k$. Given that $w_k$ are independent and identically distributed,
\begin{equation}
\condPrb{ w_k \in W }{\mathcal{F}^k} = \Prb{ w_k \in W},
\end{equation}
and so such a procedure is certainly $1$-Markovian. 

By \cite[Proposition 1]{PAJAMA2019}, there exists a $\pi \in (0,1)$ such that these methods are Exploratory, as long as the much weaker condition,
\begin{equation}
\forall v \in \row(M) \setminus \lbrace 0 \rbrace,~ \Prb{ v'M'w_0 = 0 } < 1,
\end{equation}
is satisfied. Most common procedures (e.g., randomized Kaczmarz variants, randomized Coordinate Descent variants) will readily satisfy this, as well as most randomized vector sketching procedures. However, there can be choices in which this is not true, in which case $y_k$ will only converge along a subspace of the row space of $M$, which we describe in \cite{PAJAMA2019} for the i.i.d. case. For the rest of this discussion, we will ignore this possibility.

Now, for the methods which are selecting rows or columns (e.g., randomized Kaczmarz variants, randomized Coordinate Descent variants), we can directly apply Corollary \ref{corollary-finite-convergence} to show that such procedures will converge with probability one and provide a worst case rate of convergence. For more general methods (i.e., vector sketches), we need to verify that Definition \ref{assumption-det-control} holds for some $g \in (0,1]$.

\begin{proposition}
Suppose $\lbrace w_k \rbrace$ are independent and identically distributed. Then, there exists a $g \in (0,1]$ such that Definition \ref{assumption-consistency} holds.
\end{proposition}
\begin{proof}
It is easy to verify that the procedure is $1$-Markovian, and that the procedure is Exploratory by \cite[Proposition 1]{PAJAMA2019}. Moreover, since $\lbrace w_k \rbrace$ are independent and identically distributed, 
\begin{equation}
\cond{ \det( \tilde G' \tilde G ) }{\mathcal{F}^0} = \E{ \det( \tilde G' \tilde G )}.
\end{equation}
Suppose for a contradiction $\mathbb{E}[ \det ( \tilde G' \tilde G ) ] = 0$. Then, $ \det( \tilde G' \tilde G) = 0$ with probability one, which implies $\tilde G$ is not full column rank. This contradicts the definition of $\tilde{G}$. The conclusion follows.
\end{proof}

In light of this result, we can apply Theorem \ref{theorem-control} and Corollary \ref{corollary-convergence} to conclude, there exist a $\pi \in (0,1)$, $g \in (0,1]$, a $\delta \in (1-g,1)$, and  finite stopping time $L$ such that, for $\ell \geq L$,
\begin{equation}
\norm{ y_{\tau_{\ell} }}_2^2 \leq \delta^{\ell} \norm{y_0}_2^2,
\end{equation}
where $\mathbb{E}[\tau_{\ell}] \leq \ell( 1 + \rnk (A) / \pi )$. In other words, randomized Kaczmarz variants, randomized Coordinate Descent, and (most interesting) random vector sketching methods with independent and identically distributed values of $\lbrace w_k \rbrace$ will converge with probability one and will do so geometrically along a well-controlled subsequence (asymptotically).

\subsection{Greedy Deterministic}

We now consider methods in which $w_k$ is selected to maximize some outcome given $\mathcal{F}^{k}_1$. For row-action solvers, such methods include selecting the equation with the largest absolute residual \cite{motzkin1954} or greatest distance to a hyperplane defined by a given equation \cite{agmon1954}. For column-action solvers, such methods include selecting the equation with the largest absolute residual for the normal system \cite[\S 3.3.2]{sardy2000}. For these methods, $\zeta_k$ can generally be just the empty set.

Of course, both sets of methods can be generalized by choosing $w_k$ by maximizing some distance or residual function over a finite set of choices. For a row-action example, let $w_k$ be selected from a set of basis vectors, $\lbrace h_1,\ldots,h_n \rbrace$, for $\mathbb{R}^n$ such that
\begin{equation}
w_k \in \argmax_{h \in \lbrace h_1,\ldots,h_n \rbrace} | h'(Ax_k - b) |,
\end{equation}
where we choose the smallest index in case of ties. This example row-action method generalizes the maximum residual method described in \cite{motzkin1954}.

Analogously, for a column-action example, let $w_k$ be selected from a set of basis vector, $\lbrace c_1,\ldots,c_d\rbrace$, for $\mathbb{R}^d$ such that
\begin{equation}  \label{procedure-greedy-col-distance}
w_k \in \argmax_{\substack{c \in \lbrace c_1,\ldots,c_d \rbrace \\ Ac \neq 0}} \frac{|c'A'(Ax_k - b)|}{\norm{Ac}_2},
\end{equation}
where we choose the smallest index in case of ties. This example column-action method is a generalization of \cite[\S 3.3.2]{sardy2000} in the flavor of \cite[\S 3]{agmon1954}.

We demonstrate how to apply our result on this latter example. First, such a method is $1$-Markovian given that the search direction $w_k$ depends only on $x_k$. Second, the following proposition demonstrates that this procedure is exploratory.

\begin{proposition}
The $1$-Markovian procedure described in \eqref{procedure-greedy-col-distance} is exploratory with $\pi = 1$ for a consistent system.
\end{proposition}
\begin{proof}
For a $1$-Markovian procedure, we need to verify, $\exists \pi \in (0,1]$ such that
\begin{equation}
\sup_{y_0 \in \clmn(A)\setminus \lbrace 0 \rbrace } \condPrb{ w_0'A' y_0 = 0}{\mathcal{F}^0} \leq 1 - \pi.
\end{equation}
Since $y_0 = Ax_0 - b$ and the system is consistent, this is equivalent to verifying, $\exists \pi \in (0,1]$ such that
\begin{equation}
\sup_{ x_0 : Ax_0 \neq b } \condPrb{ w_0' A'(Ax_0 - b) = 0}{\mathcal{F}^0} \leq 1 - \pi.
\end{equation}
Suppose now $w_0 ' A' (Ax_0 - b) = 0$. Then, for all $c_i' A'(Ax_0 - b) = 0$ for all $i =1,\ldots,d$. Since $\lbrace c_1,\ldots,c_d \rbrace$ form a basis, then we conclude $w_0' A' (Ax_0 - b) = 0$ if and only if $A' (Ax_0 - b) = 0$. In other words, $Ax_0 - b \in \clmn(A) \cap \clmn(A)^\perp$, which contradicts $Ax_0 \neq b$. Hence, $\condPrb{ w_0' A'(Ax_0 - b) = 0 }{\mathcal{F}^0} = 0$ for any $Ax_0 \neq b$.
\end{proof}

By Corollary \ref{corollary-finite-convergence}, the procedure described in \eqref{procedure-greedy-col-distance} produces iterates that converge to a solution with probability one, and there is a uniform rate of convergence over the subsequence $\lbrace \norm{y_{\tau_{\ell}}}_2 \rbrace$. What is more, $\tau_{\ell} \leq \ell (1 + \rnk(A))$, given that the procedure is deterministic. Nearly identical results can be derived for the other aforementioned procedures.

\subsection{Deterministic and Random Cyclic} 

Consider a procedure in which $\lbrace w_k \rbrace$ deterministically or randomly cycle through a finite set, regardless of the iterates. There are several ways of construing such procedures within our framework. To illustrate, consider a row-action solver that cycles through the rows of a simple $3 \times 3$ system. One way of encapsulating this procedure is to start with $\zeta_{-1}$ as a permutation of the indices of the three equations. Then, we choose $w_0$ to be the equation corresponding to the first entry of the permutation, and we let $\zeta_0$ be the remaining two elements of the permutation. Then, we choose $w_1$ to be the equation corresponding to the first entry of the remaining permutation, and let $\zeta_{1}$ to be the final entry in the permutation. At the next iteration, we select $w_{2}$ to be the equation corresponding to this final entry in the permutation, and let $\zeta_{2}$ be repeated as the original permutation (for a deterministic method) or a random permutation. The procedure then continues in the same pattern.
This first way of encapsulating the procedure is $1$-Markovian given that as long as $\zeta_{k-1}$ is available, $w_k$ and $\zeta_k$ can be selected. 

A second way of encapsulating the procedure is as a $3$-Markovian method. For a deterministic procedure, if $w_{k-2}, w_{k-1}$ then $w_k$ is determined. For a random permutation procedure, $w_0$ is selected randomly from the equations; given $w_0$, $w_1$ is selected from the remaining equations; given $w_0$ and $w_{1}$, $w_2$ is selected to be the remaining equation; $w_3$ is then selected randomly from the equations, regardless of $w_1$ and $w_2$; and the procedure proceeds logically. 

The choice of $N$ in the encapsulation will have an impact on the value of $\pi$ in Definition \ref{assumption-nontrivial-decay}, and, consequently, on the bounds on $\E{\tau_k}$ through the term $N/\pi$ (see Theorem \ref{theorem-control}). For example, if the rows of the $3 \times 3$ matrix in our example are orthogonal, then the $1$-Markovian encapsulation \textit{is not exploratory}. On the other hand, in the $3$-Markovian encapsulation, the procedure is exploratory with $\pi = 1$. Intermediate encapsulations can provide other values of $\pi$. 

For a simple, broad discussion, suppose $\lbrace w_k \rbrace$ cycles through a finite set of size $q$, and let us specify it as a $q$-Markovian method. So long as the span of $M'$ times the vectors in this finite set is equal to $\row(M)$, then the procedure is exploratory with $\pi = 1$, and Corollary \ref{corollary-finite-convergence} can be applied to demonstrate convergence and compute a worse case rate of convergence over a subsequence (though, the bounds on the stopping times can be easily improved). If the span of $M'$ times the finite set is not equal to $\row(M)$, then the procedure will only converge in a subspace of $\row(M)$, which is described in \cite{PAJAMA2019}.

\subsection{Greedy Subsets followed by Randomization}

Consider a procedure akin to greedy deterministic methods that choose multiple equations according to some rule depending on $\mathcal{F}_1^k$, and then randomly selecting from this subset. An example of such a procedure is given in \cite{bai2018}, in which the subset of equations selected have absolute residuals that exceed a threshold that depends on the total residual norm, and then an equation is randomly selected from this subset (with probabilities based on relative residuals). 
For this example, the selected subset is non-empty and contains equations whose residuals are nonzero. Therefore, just as for greedy deterministic procedures, this selection procedure is $1$-Markovian and exploratory with $\pi = 1$. 

Of course, simple generalizations of \cite{bai2018} include changing the basis as we discussed for greedy deterministic procedures, allowing arbitrary probabilities of selection within the greedy subset, cycling through the greedy subset, or changing the threshold function so long as the resulting greedy subset is nonempty. Importantly, our results cover all of these generalizations.

To illustrate, consider a procedure that (1) chooses the greedy subset to be the ten equations with the largest absolute residuals at the current iterate, where ties are broken by choosing the equation with the smallest index; (2) randomly selects an equation from this subset with uniform probability; and (3) updates the iterate using the resulting equation. Clearly, such a procedure is $1$-Markovian, and we now verify that it is exploratory.

\begin{proposition}
For a consistent system, the $1$-Markovian procedure just described is exploratory with $\pi = 1/10$ and uniformly controlled in expectation for some $g \in (0,1]$. 
\end{proposition}
\begin{proof}
By Definition \ref{assumption-nontrivial-decay}, we need to verify that there is a $\pi \in (0,1]$ such that
\begin{equation}
\sup_{ x_0 : Ax_0 \neq b} \condPrb{ w_0'(Ax_0 - b) = 0 }{\mathcal{F}^0} \leq 1 - \pi. 
\end{equation}
By construction, at least one equation must have a nonzero residual (else, $Ax_0 = b$). Thus, the subset of ten equations selected by the procedure has at least one equation whose residual is nonzero. The probability of failing to select this equation from a uniform distribution is no worse than $9/10$. Hence, $\pi = 1/10$.
\end{proof}

With these facts in place, Corollary \ref{corollary-finite-convergence} can be applied to demonstrate convergence and compute a worse case rate of convergence over a well-controlled subsequence. Of course, similar results can be derived for column-action methods. 

\subsection{Random Subsets followed by Greedy Selection} 

Consider a procedure that chooses a random subset of equations and then greedily selects an equation from this subset. An example of such a procedure is given in \cite{haddock2019}, in which a random subset of a fixed size is selected by uniform sampling without replacement, and then the equation with the largest absolute residual within this subset is chosen. Simple generalizations include generating the sampled subset with non-uniform distributions, replacing the residual with maximum distance, or cycling through the randomized subset. We can also generate analogues that correspond to column action methods.

Owing to the greedy component of these procedures, they are $1$-Markovian. Moreover, so long as $x_0$ does not solve the system and the procedure has a nonzero probability of selecting a given equation, the procedure is exploratory and $\pi > 0$ can be determined from the specific sampling procedure. Following the same recipe as before, Corollary \ref{corollary-finite-convergence} can be applied to demonstrate convergence and compute a worse case rate of convergence over a well-controlled subsequence. 

\subsection{Streaming Equations} 

Our results will also apply to problems in which the equations of the system are generated in a streaming fashion. To set up the streaming problem, suppose there is a random variable $\alpha \in \mathbb{R}^d$ and a random variable $\beta \in \mathbb{R}$ such that $\lbrace x : \Prb{ \alpha'x = \beta } = 1 \rbrace \neq \emptyset$. Now, suppose we observe $\lbrace (\alpha_k,\beta_k) \rbrace$, which are an independent sequence with identical distribution to $(\alpha,\beta)$, from which we want to find a solution in the set $\lbrace x : \Prb{ \alpha'x = \beta} = 1 \rbrace$. This is referred to as the streaming problem. 

Let $x_0 \in \mathbb{R}^d$ and $x^*$ denote its projection onto $\lbrace x: \Prb{ \alpha'x = \beta } = 1 \rbrace$. Within our framework, we consider the trivial linear problem of solving $x = x^*$, and we let $w_k = \alpha_k$. Then, the update is given by
\begin{equation}
x_{k+1} = x_k + \frac{\alpha_k \alpha_k'(x^* - x_k)}{\norm{\alpha_k}_2^2} = x_k + \frac{\alpha_k(\beta_k - \alpha_k'x_k)}{\norm{\alpha_k}_2^2}.
\end{equation}

This procedure is $1$-Markovian. The following proposition ensures that the streaming procedure is exploratory.

\begin{proposition}
Suppose $\E{\alpha \alpha'}$ is positive definite. Then the streaming problem is exploratory for some $\pi \in (0,1]$.
\end{proposition}
\begin{proof}
Suppose
\begin{equation}
1 = \Prb{ \alpha'(x_0- x^*) = 0 } = \Prb{ [\alpha'(x_0 - x^*)]^2 = 0 }.
\end{equation}
Then, 
\begin{equation}
0 = \E{ (x_0 - x^*)'\alpha\alpha'(x_0 - x^*)} = (x_0 - x^*)'\E{ \alpha \alpha' } (x_0 - x^*) \succ 0,
\end{equation}
which is a contradiction. Hence, for every $(x_0, x^*)$ pair, $\Prb{ \alpha'(x_0 - x^*) = 0} < 1$. By \cite[Proposition 1]{PAJAMA2019}, the conclusion follows.
\end{proof}

Just as for the independent and identically distributed procedures, there exists a $g \in (0,1]$ such that the streaming problem is uniformly controlled in expectation. The conclusion follows by Corollary \ref{corollary-convergence}.

\section{Conclusion}
\label{section-conclusion}
Motivated by ease of developing highly customized (deterministic and random) row-action and column-action solvers, we developed a general theory of with probability one convergence and worst-case rate of convergence for such methods under rather reasonable properties: $N$-Markovian, Exploratory, and either finiteness or uniformly controllable in expectation. Moreover, we demonstrated how to verify these properties for a wide array of procedures in the literature, and offered many generalizations. Therefore, we have given practitioners a set of properties to guide the design of adaptive, deterministic or random, row-action or column-action solvers that are tailored to the unique problem structure and hardware context and that are guaranteed to converge. In future work, we will address block adaptive procedures.

\bibliographystyle{amsplain}
\bibliography{bibliography}

\providecommand{\bysame}{\leavevmode\hbox to3em{\hrulefill}\thinspace}
\providecommand{\MR}{\relax\ifhmode\unskip\space\fi MR }
\providecommand{\MRhref}[2]{%
  \href{http://www.ams.org/mathscinet-getitem?mr=#1}{#2}
}
\providecommand{\href}[2]{#2}
\begin{thebibliography}{10}

\bibitem{agmon1954}
Shmuel Agmon, \emph{The relaxation method for linear inequalities}, Canadian
  Journal of Mathematics \textbf{6} (1954), 382--392.

\bibitem{bai2018}
Zhong-Zhi Bai and Wen-Ting Wu, \emph{On greedy randomized {Kaczmarz} method for
  solving large sparse linear systems}, SIAM Journal on Scientific Computing
  \textbf{40} (2018), no.~1, A592--A606.

\bibitem{censor1981}
Yair Censor, \emph{Row-action methods for huge and sparse systems and their
  applications}, SIAM Review \textbf{23} (1981), no.~4, 444--466.

\bibitem{du2020}
Yi-Shu Du, Ken Hayami, Ning Zheng, Keiichi Morikuni, and Jun-Feng Yin,
  \emph{Kaczmarz-type inner-iteration preconditioned flexible {GMRES} methods
  for consistent linear systems}, arXiv preprint arXiv:2006.10818 (2020).

\bibitem{gordon1970}
Richard Gordon, Robert Bender, and Gabor~T Herman, \emph{Algebraic
  reconstruction techniques ({ART}) for three-dimensional electron microscopy
  and x-ray photography}, Journal of Theoretical Biology \textbf{29} (1970),
  no.~3, 471--481.

\bibitem{gower2015}
Robert~M Gower and Peter Richtárik, \emph{Randomized iterative methods for
  linear systems}, SIAM Journal on Matrix Analysis and Applications \textbf{36}
  (2015), no.~4, 1660--1690.

\bibitem{haddock2019}
Jamie Haddock and Anna Ma, \emph{Greed works: An improved analysis of sampling
  {Kaczmarz-Motkzin}}, arXiv preprint arXiv:1912.03544 (2019).

\bibitem{karczmarz1937}
S~Karczmarz, \emph{Angenaherte auflosung von systemen linearer glei-chungen},
  Bull. Int. Acad. Pol. Sic. Let., Cl. Sci. Math. Nat. (1937), 355--357.

\bibitem{lent1976}
A~Lent, \emph{Maximum entropy and multiplicative {ART}}, Proc. Conf. Image
  Analysis and Evaluation, SPSE, Toronto, 1976.

\bibitem{leventhal2010}
Dennis Leventhal and Adrian~S Lewis, \emph{Randomized methods for linear
  constraints: convergence rates and conditioning}, Mathematics of Operations
  Research \textbf{35} (2010), no.~3, 641--654.

\bibitem{ma2015}
Anna Ma, Deanna Needell, and Aaditya Ramdas, \emph{Convergence {Properties} of
  the {Randomized} {Extended} {Gauss-Seidel} and {Kaczmarz} {Methods}}, SIAM
  Journal on Matrix Analysis and Applications \textbf{36} (2015), no.~4,
  1590--1604.

\bibitem{meany1969}
R.~K. Meany, \emph{A matrix inequality}, SIAM Journal on Numerical Analysis
  \textbf{6} (1969), no.~1, 104--107.

\bibitem{mills2020}
Richard~Tran Mills, Mark~F Adams, Satish Balay, Jed Brown, Alp Dener, Matthew
  Knepley, Scott~E Kruger, Hannah Morgan, Todd Munson, Karl Rupp, et~al.,
  \emph{Toward performance-portable {PETSc} for {GPU}-based exascale systems},
  arXiv preprint arXiv:2011.00715 (2020).

\bibitem{motzkin1954}
Theodore~Samuel Motzkin and Isaac~Jacob Schoenberg, \emph{The relaxation method
  for linear inequalities}, Canadian Journal of Mathematics \textbf{6} (1954),
  393--404.

\bibitem{needell2016}
Deanna Needell, Nathan Srebro, and Rachel Ward, \emph{Stochastic gradient
  descent, weighted sampling, and the randomized {Kaczmarz} algorithm},
  Mathematical Programming \textbf{155} (2016), no.~1-2, 549--573.

\bibitem{nutini2016}
Julie Nutini, Behrooz Sepehry, Issam Laradji, Mark Schmidt, Hoyt Koepke, and
  Alim Virani, \emph{Convergence rates for greedy {Kaczmarz} algorithms, and
  faster randomized {Kaczmarz} rules using the orthogonality graph}, arXiv
  preprint arXiv:1612.07838 (2016).

\bibitem{PAJAMA2019}
Vivak Patel, Mohammad Jahangoshahi, and Daniel~Adrian Maldonado, \emph{An
  implicit representation and iterative solution of randomly sketched linear
  systems}, arXiv preprint arXiv:1904.11919 (2019).

\bibitem{richtarik2020}
Peter Richt{\'a}rik and Martin Tak{\'a}c, \emph{Stochastic reformulations of
  linear systems: algorithms and convergence theory}, SIAM Journal on Matrix
  Analysis and Applications \textbf{41} (2020), no.~2, 487--524.

\bibitem{saad2003}
Yousef Saad, \emph{Iterative methods for sparse linear systems}, vol.~82, SIAM,
  2003.

\bibitem{sameh1978}
Ahmed~H Sameh and David~J Kuck, \emph{On stable parallel linear system
  solvers}, Journal of the ACM (JACM) \textbf{25} (1978), no.~1, 81--91.

\bibitem{sardy2000}
Sylvain Sardy, Andrew~G Bruce, and Paul Tseng, \emph{Block coordinate
  relaxation methods for nonparametric wavelet denoising}, Journal of
  computational and graphical statistics \textbf{9} (2000), no.~2, 361--379.

\bibitem{steinerberger2020}
Stefan Steinerberger, \emph{A weighted randomized kaczmarz method for solving
  linear systems}, arXiv preprint arXiv:2007.02910 (2020).

\bibitem{strohmer2009}
Thomas Strohmer and Roman Vershynin, \emph{A randomized {Kaczmarz} algorithm
  with exponential convergence}, Journal of Fourier Analysis and Applications
  \textbf{15} (2009), no.~2, 262.

\bibitem{wallace2014}
Tim Wallace and Ali Sekmen, \emph{Deterministic versus randomized {Kaczmarz}
  iterative projection}, arXiv preprint arXiv:1407.5593 (2014).

\bibitem{zouzias2013}
Anastasios Zouzias and Nikolaos~M Freris, \emph{Randomized extended {Kaczmarz}
  for solving least squares}, SIAM Journal on Matrix Analysis and Applications
  \textbf{34} (2013), no.~2, 773--793.

\end{thebibliography}

\end{document}